\documentclass[10pt]{amsart}
\input epsf

\usepackage{amsthm,amsmath,amssymb,latexsym}
\usepackage[all]{xy}

\begin{document}

\author{Hrant Hakobyan and Dragomir \v Sari\' c}
\thanks{H.~H. was partially supported by Kansas NSF EPSCoR Grant NSF68311}
\thanks{D.~S. was partially supported by National Science Foundation
grant DMS 1102440 and by the Simons Foundation Collaboration Grant for
Mathematicians 2011.}

\address{HH:Department of Mathematics, Kansas State University, Manhattan, KS, 66506}

\email{hakobyan@math.ksu.edu}

\address{DS:Department of Mathematics, Queens College of CUNY,
65-30 Kissena Blvd., Flushing, NY 11367} \email{Dragomir.Saric@qc.cuny.edu}

\address{DS:Mathematics PhD. Program, The CUNY Graduate Center, 365 Fifth Avenue, New York, NY 10016-4309}

\theoremstyle{definition}

 \newtheorem{definition}{Definition}[section]
 \newtheorem{remark}[definition]{Remark}
 \newtheorem{example}[definition]{Example}

\newtheorem*{notation}{Notation}

\theoremstyle{plain}

 \newtheorem{proposition}[definition]{Proposition}
 \newtheorem{theorem}[definition]{Theorem}
 \newtheorem{corollary}[definition]{Corollary}
 \newtheorem{lemma}[definition]{Lemma}

\newcommand{\eps}{\varepsilon}
\newcommand{\G}{\Gamma}
\newcommand{\g}{\gamma}
\newcommand{\D}{\Delta}
\renewcommand{\d}{\delta}
\newcommand{\dist}{\mathrm{dist}}
\newcommand{\m}{\mathrm{mod}}

\title{Vertical limits of graph domains}


\date{\today}

\maketitle

\begin{abstract}
We consider the limiting behavior of Teichm\"uller geodesics in the
universal Teichm\"uller space $T(\mathbb{H})$. Our main result states that
the limits of the Teichm\"uller geodesics in the Thurston's boundary of
$T(\mathbb{H})$ may depend on both vertical and horizontal foliation of the
corresponding holomorphic quadratic differential.
\end{abstract}

\section{Introduction}

By Uniformization Theorem, a simply connected domain $D$ in the complex plane
$\mathbb{C}$ is conformally equivalent to the hyperbolic plane $\mathbb{H}$.
The set of prime ends of $D$ is homeomorphic to the unit circle
$\mathbb{S}^1$-the ideal boundary of $\mathbb{H}$ (cf. \cite{Pomm}). Unless
stated otherwise, we implicitly assume this identification.

The map $T_{\eps}$ of $D$ that is obtained by multiplying the distances in
the vertical direction by $\eps >0$ is called the {\it Teichm\"uller map}.
Thus, for $\eps>0$ we have
$$T_{\eps}(x,y)=(x,\eps y).$$

The image of $D$ under $T_{\eps}$ is a new simply connected domain $D_{\eps}$
in $\mathbb{C}$. The Teichm\"uller map extends by continuity to a marking
homeomorphism between the space of prime ends of $D$ and the space of prime
ends of $D_{\eps}$. Note that both spaces of prime ends are implicitly
identified with the unit circle $\mathbb{S}^1$ (cf. \cite{Pomm}). We prove
(cf. \S 5)



\begin{theorem}\label{thm:arbitrary}
Let $D$ be a simply connected domain under the graph of a real-valued
function. Assume that $\G$ is the family of curves in $D$ connecting
$(a,b)\subset \mathbb{S}^1$ and $(c,d)\subset \mathbb{S}^1$. Then
\begin{equation}\label{equality:main}
\lim_{\eps\to 0}\eps\cdot \m(\G^{\eps})=\m(\G_{v}),
\end{equation}
where $\G^{\eps}=T_{\eps}(\G)$ is the image of $\G$ under the Teichmuller map
and $\G_{v}$  is the family, possibly empty, of {{vertical line segments}} in
$\G$.
\end{theorem}

In the theorem above $\m(\G)$ denotes the conformal modulus of a curve family
$\G$, see Section \ref{Section:modulus+Liouville} for the definition of the
modulus.

Next, we interpret Theorem \ref{thm:arbitrary} in terms of the asymptotic
behavior of Teichm\"uller geodesics corresponding to a particular type of
quadratic differentials in the universal Teichm\"uller space.

The universal Teichm\"uller space $T(\mathbb{H})$ consists of all
quasisymmetric maps of $\mathbb{S}^1$ which fix $-i,1,i\in \mathbb{S}^1$. The
Teichm\"uller map under the identification of the prime ends of $D$ and
$D^{\epsilon}$ with $\mathbb{S}^1$ induces a quasisymmetric map
$h_{\epsilon}:\mathbb{S}^1\to \mathbb{S}^1$. Thus we obtain a path in the
universal Teichm\"uller space $T(\mathbb{H})$ parameterized by $\eps
>0$ which corresponds to a Teichm\"uller geodesic. The path $\eps\mapsto
h_{\eps}$ is unbounded in the Teichmuller metric (see Section 2 for the
definition) as $\eps\to 0$.

We consider the question of finding the limiting behavior of the
Teichm\"uller geodesic in the universal Teichm\"uller space $T(\mathbb{H})$.
Masur \cite{Mas} described the limiting behavior of Teichm\"uller geodesics
for the case of compact surfaces. He showed that if the vertical foliation
 of the corresponding quadratic differential $\varphi$ is
uniquely ergodic then the limit of the Teichm\"uller geodesic in the
Thurston's boundary is the projective class of the measured lamination which
is equivalent to the vertical foliation of $\varphi$ (cf. \cite{Mas}). When
the vertical foliation consists of finitely many cylinders, then the limit is
the projective lamination with support consisting of closed geodesics
homotopic to the cylinders of the vertical foliation but the weights are all
equal while the cylinder heights might be different (cf. \cite{Mas}).
Moreover, there are examples of Teichm\"uller geodesics which do not have
unique limiting points on Thurston's boundary (cf. \cite{Len}).

In the case of closed surfaces, the limiting behavior of Teichm\"uller
geodesics is investigated using the lengths of simple closed geodesics. There
are no closed geodesics in the hyperbolic plane $\mathbb{H}$. Thurston's
boundary to the universal Teichm\"uller space $T(\mathbb{H})$ is identified
with the space $PML_{bdd}(\mathbb{H})$ of projective bounded measured
laminations on $\mathbb{H}$ (cf. \cite{Sa2}, \cite{Sa4}). The bordification
of $T(\mathbb{H})$ is done using geodesic currents, i.e. the space of
positive Borel measures on the space of geodesics of $\mathbb{H}$ (for
definition see \S 2 and \cite{Bon1}). Therefore the study of limit points
involves the study of the limits of geodesic currents. Although we consider
limits of Teichm\"uller geodesics as in the case of closed surfaces, the
ideas and arguments used are somewhat more analytical in nature and are
disjoint from the prior work on closed surfaces (cf. \cite{Mas}, \cite{Len}).
We prove

\begin{theorem}\label{thm:teichmuller}
{\it Let $\varphi :\mathbb{H}\to\mathbb{C}$ be an integrable holomorphic
quadratic differential on the hyperbolic plane $\mathbb{H}$ without zeros or
poles in $\mathbb{H}$. Assume that the image in $\mathbb{C}$ of $\mathbb{H}$
in the natural parameter of $\varphi$ is a domain $D$ bounded by the graphs
of two functions $f(x)$ and $g(x)$ {defined on an interval $I$ of the real
line}. Denote by $h_{\eps}$, for $\eps
>0$, the Teichm\"uller geodesic which scales the vertical direction of
$\varphi$ by $\eps >0$. Then the limit of the Teichm\"uller geodesic
$h_{\eps}$ as $\eps\to 0$ is equal to the projective class of the measured
lamination whose support is homotopic to the vertical foliation of $\varphi$
and whose transverse measure is given by
$$
\int_I \frac{1}{|f(x)-g(x)|}dx
$$
where $I$ is a horizontal arc transverse to the vertical foliation and $dx$
is the linear measure on the horizontal {interval $I$}. }
\end{theorem}

We remark that the limiting projective measured lamination, although unique,
cannot be described solely in terms of the vertical foliation of the
holomorphic quadratic differential $\varphi$. This is a new phenomenon which
does not appear in the Teichm\"uller spaces of compact surfaces. To
illustrate this phenomenon, assume that $D$ is the domain under the graph of
a step function. Then the transverse measure is a multiple of the linear
measure by the reciprocal of the heights of the steps.

One consequence of this phenomenon is that if we consider $\varphi$ on
$\mathbb{H}$ and a corresponding holomorphic differential $\varphi_1$ on
$f(\mathbb{H})$, where $f$ is a marking map defining a point in
$T(\mathbb{H})$, the limits of the corresponding Teichm\"uller geodesics in
the Thurston's boundary are different even though $\varphi$ and $\varphi_1$
have the same vertical foliations. In the case of closed surfaces the limits
are the same.

We point out that finite area of $D$ can be replaced by ``locally finite
area'' condition-for each finite horizontal arc the total area of the domain
formed by the vertical leaves intersecting it is finite. The convergence in
Theorem 2 is in the weak* topology on the geodesic currents. It is an
interesting question to determine whether the above convergence holds for the
uniform weak* topology from \cite{Sa3}. Moreover, it would be interesting to
extend Theorem 2 to the case of arbitrary finite area Jordan domain or even
to arbitrary integrable holomorphic quadratic differentials.

\section{Thurston's boundary of the universal Teichm\"uller space}

Let $\mathbb{H}$ be the hyperbolic plane. The ideal boundary of $\mathbb{H}$
is homeomorphic to the unit circle $\mathbb{S}^1$ in the complex plane. A
homeomorphism $h:\mathbb{S}^1\to \mathbb{S}^1$ is said to be {\it
quasisymmetric} if there exists $M\geq 1$ such that
$$
\frac{1}{M}\leq\frac{|h(I)|}{|h(J)|}\leq M
$$
for all circular arcs $I,J$ with a common boundary point and disjoint
interiors such that $|I|=|J|$, where $|I|$ is the length of $I$. A
homeomorphism is quasisymmetric if and only if it extends to a quasiconformal
map of the unit disk. Since in this note we will not use quasiconformal
mappings we refer to the classical lecture notes of Ahlfors
\cite{Ahlfors:QClectures} for background on planar quasiconformal mappings.

\begin{definition}
The universal Teichm\"uller space $T(\mathbb{H})$ consists of all
quasisymmetric maps $h:\mathbb{S}^1\to \mathbb{S}^1$ that fix $-i,1,i\in
\mathbb{S}^1$.
\end{definition}

If $g:\mathbb{D}\to\mathbb{D}$ is a quasiconformal map, denote by $K(g)$ its
quasiconformal constant. The Teichm\"uller metric on $T(\mathbb{H})$ is given
by ${d(h_1,h_2)=\inf_g \log K(g)}$, where $g$ runs over all quasiconformal
extensions of the quasisymmetric map $h_1\circ h_2^{-1}$. The Teichm\"uller
topology is induced by the Teichm\"uller metric.

Thurston \cite{Bon1},\cite{FLP},\cite{Thurston} introduced a boundary to the
Teichm\"uller space of a closed hyperbolic surface as follows. First, the
Teichm\"uller space $T(S)$ of a closed surface $S$ embeds into
$\mathbb{R}^{\mathcal{S}}$, where $\mathcal{S}$ is the set of all simple
closed curves of $S$. The embedding
$T(S)\hookrightarrow\mathbb{R}^{\mathcal{S}}$ is defined by assigning to each
$\alpha\in\mathcal{S}$ the length of its geodesic representative for the
marked hyperbolic metric on the surface $S$ defining the point of $T(S)$. The
Teichm\"uller space remains embedded after projectivization
$T(S)\hookrightarrow \mathbb{R}^{\mathcal{S}}\hookrightarrow
P\mathbb{R}^{\mathcal{S}}$ and Thurston's boundary consists of the limit
points of the image of $T(S)$. It turns out that the Thurston's boundary is
identified with the space of projective measured laminations on $S$.

Bonahon \cite{Bon1} used a different approach to obtain Thurston's boundary
by embedding $T(S)$ into the space of geodesic currents on $S$. A {\it
geodesic current} on $S$ is a positive Borel measure on the space of
geodesics $(\mathbb{S}^1\times \mathbb{S}^1\setminus diag)/\mathbb{Z}_2$ of
the universal covering $\mathbb{H}$ of $S$ that is invariant under the action
of the covering group $\pi_1(S)$. Each point in the Teichm\"uller space
$T(S)$ is a (marked) hyperbolic metric which defines a unique (up to positive
multiple) {positive measure of full support, called the {\it Liouville
measure}, on the space of geodesics of the universal covering invariant under
the action of the covering group. Since the marking maps conjugate covering
groups, the pull backs of the Liouville measures under the marking maps give
geodesics currents on the base surface $S$. Then the closure of the
projectivization of the embedding of $T(S)$ in the space of the geodesic
currents of $S$ gives Thurston's boundary \cite{Bon1}.

The approach to the Thurston's boundary using geodesic currents is used in
\cite{Sa2}, \cite{Sa3} to introduce Thurston's boundary to the Teichm\"uller
space of arbitrary hyperbolic Riemann surface including the universal
Teichm\"uller space $T(\mathbb{H})$ because arbitrary Riemann surfaces might
not have enough non-trivial closed curves (e.g. the hyperbolic plane
$\mathbb{H}$ has no non-trivial closed curves). The space of geodesics of the
hyperbolic plane $\mathbb{H}$ is identified with $\mathbb{S}^1\times
\mathbb{S}^1\setminus diag$ by assigning to each geodesic the pair of its
endpoints.
 The {\it Liouville measure} $\mathcal{L}$ on the space of geodesic of
$\mathbb{H}$ is given by
$$
 \mathcal{L}(A)=\int_A \frac{d\alpha d\beta}{|e^{i\alpha}-e^{i\beta}|^2}
$$
for any Borel set $A\subset \mathbb{S}^1\times \mathbb{S}^1$. If
$A=[a,b]\times [c,d]$ then
$$
\mathcal{L}([a,b]\times [c,d])=\log\frac{(a-c)(b-d)}{(a-d)(b-c)}.
$$

To each $h\in T(\mathbb{H})$, we assign the pull-back $h^{*}(\mathcal{L})$ of
the Liouville measure by the quasisymmetric map $h:\mathbb{S}^1\to
\mathbb{S}^1$. This assignment is a homeomorphism of $T(\mathbb{H})$ onto its
image in the space of bounded geodesic currents; a geodesic current $\alpha$
is {\it bounded} if
$$
\sup_{[a,b]\times [c,d]}\alpha ([a,b]\times [c,d])<\infty
$$
where the supremum is over all $[a,b]\times [c,d]$ with
$\frac{(a-c)(b-d)}{(a-d)(b-c)}=2$. The space of bounded geodesic currents is
endowed with the family of H\"older norms parametrized with the H\"older
exponents $0<\nu\leq 1$ (cf. \cite{Sa2}). The homeomorphism of
$T(\mathbb{H})$ into the space of bounded geodesic currents is differentiable
with a bounded derivative given by a H\"older distribution (cf. \cite{Sa4})
and, in fact, Otal \cite{Ot} proved that it is real-analytic. The map from
$T(\mathbb{H})$ to the projective bounded geodesic currents remains a
homeomorphism and the boundary points  of the image of $T(\mathbb{H})$ are
all projective bounded measured laminations (cf. \cite{Sa2}). Thus Thurston's
boundary of $T(\mathbb{H})$ is the space $PML_{bdd}(\mathbb{H})$ of all
projective bounded measured laminations on $\mathbb{H}$ (and an analogous
statement holds for any hyperbolic Riemann surface). Alternatively, the space
of geodesic currents can be endowed with the uniform weak* topology and
Thurston's boundary for $T(\mathbb{H})$ is again $PML_{bdd}(\mathbb{H})$ (cf.
\cite{Sa3}).

\section{The limits of the moduli of families of curves and \\ the Liouville
measure}\label{Section:modulus+Liouville}

{Let $R$ be a simply connected region in $\mathbb{C}$ other than the complex
plane.} Let $f:R\to\mathbb{H}$ be the Riemann mapping, where $\mathbb{H}$ is
the unit disk model of the hyperbolic plane. Then the set of prime ends of
$R$ in the sense of Caratheodory is in a one to one correspondence with the
points of the unit circle $\mathbb{S}^1$ (cf. \cite{Pomm}). When we consider
a simply connected domain we will always implicitly assume the correspondence
of the prime ends with the points of the unit circle $\mathbb{S}^1$ under the
Riemann mapping.

Our goal is to relate the Liouville measure associated to two closed disjoint
arcs of $\mathbb{S}^1$ with the modulus of the family of curves (defined
below) in $\mathbb{H}$ connecting the two closed arcs. The correspondence
between the prime ends of a simply connected domain $R$ and $\mathbb{S}^1$
directly translates to $R$ the conclusions that we obtain for $\mathbb{S}^1$.

\subsection{Conformal modulus and its properties.}
Next we define the conformal modulus of a family of curves in $\mathbb{C}$
which is the main tool in this note. Suppose $\G$ is a family of locally
rectifiable curves in $\mathbb{C}$. A non-negative Borel measurable function
$\rho:\mathbb{C}\to[0,\infty]$ is called a $\G$ - \textit{addmissible metric}
 if for every $\gamma\in\Gamma$ we have
$$
l_{\rho}(\gamma )=\int_{\gamma}\rho (z)|dz|\geq 1.
$$
The quantity  $l_{\rho}(\gamma )$ is often called the $\rho$-length of $\g$.
The \textit{conformal modulus} $\m(\Gamma)$ of $\Gamma$ is defined by
$$
\m(\Gamma)=\inf_{\rho}\int_{\mathbb{D}}\rho(z)^2dxdy
$$
where the infimum is over all $\G$-admissible metrics $\rho$.

In what follows we will need some basic properties of the modulus. We refer
 to \cite{LV,Vaisala:lectures} for the proofs of the properties
below and for further background on conformal modulus.}

We will say that $\G_1$ \emph{overflows} $\G_2$ and will write $\G_1>\G_2$ if
every curve $\g_1\in \G_1$ contains some curve $\g_2\in \G_2$.

\begin{lemma} Let $\G_1,\G_2,\ldots$ be curve families in $\mathbb{C}$. Then
\begin{itemize}
  \item[1.] (\textbf{Monotonicity}) If $\G_1\subset\G_2$ then $\m(\G_1)\leq
      \m(\G_2)$.
  \item[2.] (\textbf{Subadditivity}) $\m(\bigcup_{i=1}^{\infty} \G_i) \leq
      \sum_{i=1}^{\infty}\m(\G_i).$
  \item[3.] (\textbf{Overflowing}) If $\G_1<\G_2$ then $\m (\G_1) \geq \m
      (\G_2)$.
\end{itemize}
\end{lemma}

Another very important property of the conformal modulus is its invariance
under conformal mappings of the plane.

\begin{lemma}[Conformal invariance of modulus]
Suppose $\G$ is a family of curves in a domain $D\subset\mathbb{C}$ and $f$
is a conformal mapping of $D$ onto $D'$. Let $f(\G)\subset D'$ be the image
of the family $\G$, i.e. $f(\G)=\{f(\g): \g\in\G\}$. Then
  $$\m(f(\G)) = \m(\G).$$
\end{lemma}

\subsection{Modulus and Liouville measure}

Let $(a,b,c,d)$ be a quadruple of distinct points on $\mathbb{S}^1$ given in
the counterclockwise order. Denote by $\Gamma_{(a,b,c,d)}$ the family of all
{{locally rectifiable} curves $\gamma\subset\mathbb{D}$ connecting $(a,b)$ to
$(c,d)$, i.e. $(a,b)\cup \g \cup (c,d)$ is a connected subset of the plane.

\begin{lemma}
\label{lem:mod_liouville_measure} Let $(a,b,c,d)$ be a quadruple of points on
$\mathbb{S}^1$ in the counterclockwise order. Let $\Gamma_{(a,b,c,d)}$
consist of all curves $\gamma$ in $\mathbb{D}$ which connect $(a,b)\subset
\mathbb{S}^1$ with $(c,d)\subset \mathbb{S}^1$. Then
$$
\m(\Gamma_{(a,b,c,d)})-\frac{1}{\pi}\mathcal{L}([a,b]\times [c,d])-\frac{2}{\pi}\log 4\to 0
$$
as $mod(\Gamma_{(a,b,c,d)})\to\infty$, where $\mathcal{L}$ is the Liouville
measure.
\end{lemma}

\begin{remark}
Note that simultaneously $\m(\Gamma_{(a,b,c,d)})\to\infty$ and
$\mathcal{L}([a,b]\times [c,d])\to\infty$.
\end{remark}

\begin{proof}Consider a conformal mapping of the unit disc onto
the upper half plane mapping the points $a,b,c,d\in\mathbb{S}$ to
$w_1,w_2,w_3,\infty \in \mathbb{R}\cup\{\infty\}$, respectively, so that
$-\infty<w_1<w_2<w_3<\infty$. By the conformal invariance of the modulus we
have
\begin{equation}\label{modmobius}
\m (\G_{(a,b,c,d)}) = \m
(\Omega_{(w_1,w_2,w_3,\infty)})
\end{equation}
where $\Omega_{(w_1,w_2,w_3,\infty )}$ is the family of curves connecting the
segment $[w_1,w_2]$ to the segment $[w_3,\infty ]$ in in the upper half
plane. Furthermore, a simple symmetry argument (cf. \cite[page 81]{LV}) shows
that the modulus of the family of arcs $\Omega_{(w_1,w_2,w_3,\infty )}$
connecting the segment $[w_1,w_2]$ to the segment $[w_3,\infty ]$ in
$\mathbb{C}$ satisfies
\begin{equation}\label{modconf}
\m( \Omega_{(w_1,w_2,w_3,\infty )})=\frac{2}{\pi}\mu (\sqrt{\frac{w_3-w_2}{w_3-w_1}}),
\end{equation}
where $\mu (r)$ is the $2\pi$-multiple of the modulus of the family of closed
curves in the unit disk, which separate the unit circle $\mathbb{S}^1$ and
the arc on the real axis from $0$ to $r$ with $0<r<1$ (cf. \cite[page
53]{LV}). Careful estimates on $\mu(r)$ then yield the following asymptotics:
\begin{equation}\label{modasymptotics}
\mu (r)-\log \frac{4}{r}\to 0,
\end{equation}
as $r\to 0$ (cf. \cite[page 62, (2.11)]{LV}). Let $w_1<w_2<w_3$ be three real
numbers.

The lemma now follows easily if we combine (\ref{modmobius}),(\ref{modconf})
and (\ref{modasymptotics}) together with the fact that Liouville measure is
invariant under M\"obius maps and therefore $\mathcal{L}([w_1,w_2]\times
[w_3,\infty ])=\log\frac{w_3-w_1}{w_3-w_2}$.
\end{proof}

\section{Modulus of vertical families}

 Theorem
\ref{thm:arbitrary} states that the limiting behaviour of moduli of certain
families of curves is completely determined by the subfamily $\G_v$ of
vertical curves in $D$. For this reason we start by calculating the modulus
of a general family of vertical curves in $\mathbb{C}$.

\begin{lemma}\label{lemma:vertical}
Let $E\subset\mathbb{R}$ be a measurable set and let $\G_v$ be a family of
vertical intervals $\{\g(x)\}_{x\in E}$, where $\g(x)$ is an interval of
length $|\g(x)|>0$ contained in the vertical line passing through $x\in E
\subset\mathbb{R}$. Then the modulus of $\G_v$ can be computed using the
following Lebesgue integral
  \begin{eqnarray}
    \m (\G_v)= \int_{E} \frac{dx}{|\g(x)|}.
  \end{eqnarray}
\end{lemma}
\begin{proof}
Define
\begin{equation}\label{extremalmetric}
  \rho_0(x,y)=
  \begin{cases}
    |\g(x)|^{-1}, &\mbox{ for  }  (x,y) \in \g(x),\\
    0, &\mbox{ otherwise}.
  \end{cases}
\end{equation}
Since every $\g(x)\in\G_v$ is a vertical interval, we have that $\int_{\g(x)}
\rho_0(x,y) |dz| = |\g(x)|^{-1} \int_{\g(x)} dy = 1.$ Thus $\rho_0$ is
admissible for $\G_v$ and we have

\begin{eqnarray*}
\m \G_v
&\leq& \iint_{D} \rho_0^{2}(x,y){dxdy}
= \int_E \left(\int_{\g(x)} |\g(x)|^{-2} dy \right) {dx}  \\
&=& \int_E |\g(x)| \cdot {|\g(x)|^{-2}}{dx} = \int_E \frac{dx}{|\g(x)|}.
\end{eqnarray*}

To obtain the opposite inequality we will use the following well known
criterion of Beurling, cf. Theorem 4.4 in \cite{Ahlfors:Confinvariants}. Note
that in \cite{Ahlfors:Confinvariants} the criterion is formulated for the
extremal length rather than the modulus, but it is easily seen that the
formulation below is equivalent to the one in \cite{Ahlfors:Confinvariants}.
Recall that a $\G$-admissible metric $\rho_0$ is said to be extremal for the
family $\G$ if $\m(\G)=\iint_{D} \rho_0(x,y)^2 dxdy.$

\begin{lemma}[Beurling's criterion]
  The metric $\rho_0$ is extremal for $\G$ if there is a subfamily $\G_0\subset\G$ such that
  \begin{itemize}
    \item $\int_{\g}\rho_0 ds = 1, \forall \g\in\G_0$
    \item for any real valued $h$ in $D$ satisfying $\int_{\g}h ds \geq 0,
        \forall \g\in\G_0$ the following holds
    $$ \iint_{D} h\rho_0 dxdy \geq 0.$$
  \end{itemize}
\end{lemma}

As was noted above the function $\rho_0(x,y)$ defined in
(\ref{extremalmetric}) satisfies the first condition of the Beurling's
criterion. To check that the second condition is also satisfied note that for
a function $h$ in $D$ such that $\int_{\g(x)}h(x,y)dy\geq0$ for every $x\in
E$ we have
by Fubini's theorem
\begin{eqnarray*}
      \iint_{D} h(x,y)\rho_0(x,y) dxdy
  = \int_E |\g(x)|^{-1}\left[{\int_{\g(x)} h(x,y) dy}\right] dx \geq 0,
\end{eqnarray*}
since $|\g(x)|>0, \forall x\in \mathbb{R}$. Thus $\rho_0$ is  extremal for
$\G_v$ (in this case $\G_0=\G_v$).
\end{proof}

\section{The domains under the graphs of functions}

Given a function $f:(A,B)\to (0,\infty)$ it is well known that the set
\begin{equation}
D:=\{(x,y): x\in(A,B), 0<y<f(x)\}
\end{equation}
(a.k.a. hypograph of $f$) is open if and only if $f$ is a lower-semicontiuous
function. If this is the case we will call $D$ the \emph{graph domain of the
function $f$}. In this section we prove Theorem \ref{thm:arbitrary} by first
proving it in the case when $f$ is a continuous function and then by
approximating an arbitrary lower semicontinuous function by an increasing
sequence of continuous ones.

\subsection{A general estimate} To begin we establish some estimates which hold for an arbitrary
curve family $\G$ in any planar domain $D$ of finite area. To formulate our
result we will need a notation for a subfamily of ``almost vertical" curves
in $\G$. Namely, for $\eta>0$ we define subfamilies $\G_{<\eta}$ and
$\G_{\geq\eta}$ of $\G$ as follows
  \begin{eqnarray*}
    \G_{\geq\eta}&=&\{\g\in\G : |\pi_1(\g)|\geq\eta\},\\
    \G_{<\eta}&=&\{\g\in\G : |\pi_1(\g)|<\eta\},
  \end{eqnarray*}
where $|\pi_1(\g)|$ is the length of the the vertical projection of $\g$ onto
the real axis.
\begin{lemma}\label{lem:general}
  Let $D$ be an arbitrary finite area domain in $\mathbb{C}$ and $\G$ be a family of
  locally rectifiable curves in $D$. Let $T_{\eps}(x,y) = (x,\eps y)$ and $\G^{\eps} = T_{\eps} (\G)$. Then
\begin{equation}\label{modestimates}
\m (\G_v) \leq \limsup_{\eps\to 0} \eps \cdot \m(\G^{\eps}) \leq \lim_{\eta\to0^{+}} \m(\G_{<\eta}).
\end{equation}
\end{lemma}

\begin{proof}
By Lemma \ref{lemma:vertical} we have
\begin{eqnarray*}
 \m(\G_v) = \int_p^q \frac{dx}{f(x)} = \eps \int_p^q \frac{dx}{\eps f(x)}
 = \eps \cdot \m (\G^{\eps}_v),
\end{eqnarray*}
Since $\G_v^{\eps}\subset\G^{\eps}$  monotonicity of the modulus imples that
$\eps \cdot \m (\G^{\eps}_v)\leq \eps \cdot \m (\G^{\eps})$. This immediately
yields the first inequality in (\ref{modestimates}).

To prove the right inequality in  (\ref{modestimates}) note that by the
subadditivity of the modulus we have
\begin{eqnarray*}
\limsup_{\eps\to0} \eps \cdot \m (\G^{\eps})
&\leq& \limsup_{\eps\to0} \eps \cdot \m (\G^{\eps}_{\geq\eta}) + \limsup_{\eps\to0} \eps \cdot \m (\G^{\eps}_{<{\eta}}).
\end{eqnarray*}
Note that $\eps \cdot \m (\G^{\eps}_{\geq{\eta}}) \to 0$ as $\eps\to0$.
Indeed, since for every $\g\in \G^{\eps}_{\geq \eta}$ we have that the length
of $\g$ is at least $\eta$ it follows that $\rho(x)=\chi_{D^{\eps}}(x)/\eta$
is admissible for $\G^{\eps}_{\geq\eta}$ and
\begin{eqnarray*}
\eps\cdot \m (\G^{\eps}_{\geq\eta})
\leq \eps \cdot \int_{D^{\eps}} (1/\eta)^2 dxdy \leq \frac{\eps}{\eta^2} \cdot A(D^{\eps})
\leq \frac {\eps^{2}A(D)}{\eta^2}  \xrightarrow[\eps\to 0 ]{} 0,
\end{eqnarray*}
where $A(D)$ denotes the two dimensional area of a domain
$D\subset\mathbb{C}$. Thus we have that
\begin{eqnarray}\label{estimate:general}
\limsup_{\eps\to0} \eps \cdot \m (\G^{\eps})
&\leq& \limsup_{\eps\to0} \eps \cdot \m (\G^{\eps}_{<{\eta}}).
\end{eqnarray}
Now, since $\G^{\eps}_{<\eta} = T_{\eps} (\G_{<\eta})$ and $T_{\eps}$ is
$\eps^{-1}$-quasiconformal we have $\m(\G^{\eps}_{<\eta}) \leq
\eps^{-1}\m(\G_{<\eta}),$ and from inequality (\ref{estimate:general}) it
follows that
\begin{equation*}
\limsup_{\eps\to0} \eps \cdot \m (\G^{\eps})
\leq
\m (\G_{<{\eta}}).
\end{equation*}
Since the last inequality holds for every $\eta>0$ and $\m(\G_{<\eta})$ is
non-decreasing in $\eta$ we obtain the right hand side inequality in
(\ref{modestimates}) by taking $\eta$ to $0$.
\end{proof}

\begin{remark}\label{remark:nicefamilies}
Lemma \ref{lem:general} implies that to prove Theorem \ref{thm:arbitrary} and
to obtain equality (\ref{equality:main}) for a family $\G$ in $D$ it is
enough to show that the following inequality holds
\begin{equation}\label{nicefamilies}
  \m(\G_v) \geq \lim_{\eta\to0^{+}} \m(\G_{<\eta}).
\end{equation}
\end{remark}

\begin{remark}
Inequality (\ref{nicefamilies}) does not hold always. For instance let $\G$
be the collection of all the curves in the unit square $[0,1]^2$ connecting
the horizontal sides $[0,1]\times\{0\}$ and $[0,1]\times\{1\}$, excluding the
family of vertical segments $\{x\}\times[0,1], 0\leq x\leq 1.$ Then
$\m(\G_{<\eta})=1$ for every $\eta>0$ while $\m(\G_v)=\m(\emptyset) = 0$ and
the inequality (\ref{nicefamilies}) fails.
\end{remark}

\subsection{Domains under graphs of continuous functions}
Let $D$ be a finite area domain in $\mathbb{C}$ under the graph of a
continuous function $f:(A,B)\to\mathbb{R}_{\geq 0}\cup\{\infty\}$, where
$(A,B)$ could be a finite or an infinite interval including
$(A,B)=\mathbb{R}$. Let $(a,b,c,d)$ be in the given cyclic order on the
boundary $\partial D$ of the domain $D$ and let $\G$ be the family of curves
in $D$ connecting $(a,b)$ to $(c,d)$. Denote by $(p,q)$ the intersection
$(a,b)\cap (c,d)$.
\begin{theorem}
\label{thm:continuous}With the notations as above the following equalities
hold
\begin{equation}\label{equality:continuous}
\lim_{\eps\to 0}\eps \cdot \m(\G^{\eps}) =\int_p^q\frac{dx}{f(x)}=\m(\G_{v}).
\end{equation}
\end{theorem}

\begin{proof}
The second equality in (\ref{equality:continuous}) holds by Lemma
\ref{lemma:vertical}. By Remark  \ref{remark:nicefamilies} we only need to
show that
\begin{equation}\label{continuous-is-nice}
\lim_{\eta\to0^{+}} \m(\G_{<\eta})\leq \int_{p}^q
\frac{dx}{f(x)}.
\end{equation}

Now, let $p=x_0<x_1<\ldots<x_n=q$ be a partition of the interval $[p,q]$.
Then by subadditivity of the modulus we obtain
\begin{eqnarray*}
\m(\G_{<\eta})
  &\leq&  \sum_{i=0}^{n-1} \m\{\g\in\G_{<\eta} : \g(0)\in[x_i,x_{i+1}]\},
\end{eqnarray*}
Considering the rectangles
$$R_i=[x_i-\eta, x_{i+1}+\eta]\times \min_{[x_i-\eta,x_{i+1}+\eta]} f$$ we note that
every curve from the family $\{\g\in\G_{<\eta} : \g(0)\in[x_i,x_{i+1}]\}$
contains a subcurve in $R_i$ which connects the horizontal sides of the
rectangle. Therefore, by the property of overflowing, we have for every
$i=0,\ldots,n-1$ the estimate
$$\m\{\g\in\G_{<\eta} : \g(0)\in[x_i,x_{i+1}]\}
\leq \frac{x_{i+1}-x_i+2\eta}{
\min_{[x_i-\eta,x_{i+1}+\eta]} f}.$$
Summing up over $i$ we obtain
\begin{eqnarray*}
\m (\G_{<\eta})
\leq
\sum_{i=0}^{n-1}  \frac{x_{i+1}-x_i+2\eta}{\min_{[x_i-\eta,x_{i+1}+\eta]} f}.
\end{eqnarray*}
Taking $\eta$ to $0$ we obtain
\begin{eqnarray*}
  \lim_{\eta\to0^{+}} \m (\G_{<{\eta}})\leq \sum_{i=0}^{n-1}  \frac{x_{i+1}-x_i}{\min_{[x_i,x_{i+1}]} f},
\end{eqnarray*}
where the sum on the right is a Riemann sum for the integral
$\int_p^q\frac{dx}{f(x)}$ and therefore can be taken to be less than
$\int_p^q\frac{dx}{f(x)}+\delta$ for every $\delta>0$. This proves
(\ref{continuous-is-nice}) and thus the theorem.
\end{proof}

\subsection{General graph domains}
In this section, we consider a graph domain $D$ in $\mathbb{C}$ of finite
Euclidean area under the graph of an arbitrary lower semicontinuous function
$f:(A,B)\to[0,\infty]$.

%
%


We consider the set of prime ends for the domain $D$. If $D$ is conformally
mapped onto the unit disk $\mathbb{H}$ then the set of prime ends of $D$ is
in a one to one correspondence with the unit circle $\mathbb{S}^1=\partial
B_1$ (cf. \cite{Pomm}). The set of prime ends inherits an orientation from
$\mathbb{S}^1$. Let $\{ s_i\}_{i=1}^{\infty}$ be a system of cross-cuts
defining prime end $a$ of $D$. A continuous ray $r:[0,1)\to D$ is said to
have an endpoint $r(1)$ equal to the prime end $a$ if there exists $t_i<1$
such that $r((t_i,1))$ is contained in the component of $D_i=D\setminus s_i$
not containing $r(0)$ for all large $i$.

\begin{lemma}
\label{lem:prime_ends} The imprint $I(a)=\cap_i\bar{D}_i$ of a prime end $a$
of the domain $D$ lies on a vertical line, where $\{ D_i\}_i$ is the system
of subdomains of $D\setminus s_i$ defining $a$.
\end{lemma}

\begin{proof}
Assume on the contrary that $z_1,z_2\in\cap_i\bar{D}_i$ with
$Re(z_1)<Re(z_2)$. Since $D$ is the domain under the graph of $f$, it follows
that each $D_i$ is a subset of the union of vertical segments between points
of $s_i$ and the graph of $f$, or the union of the vertical segments between
points of $s_i$ and the real axis. In either case, the vertical projection
onto the $x$-axis of each $s_i$ contains the interval $(Re(z_1),Re(z_2))$.
Thus the length of $s_i$ does not converge to $0$. Contradiction.
\end{proof}

We denote by $\pi_1(a)$ the $x$-coordinate of the prime end $a$. Let us
assume that the interval $(a,b)$ is on the bottom side of $D$ (belongs to
$x$-axis) and the interval $(c,d)$ is on the top side of $D$ (lies on the
part of the boundary of $D$ above the $x$-axis determined by the graph of
$f$) and $\pi_1(c)>\pi_1(d)$ (since $a,b,c,d$ are given in a cyclic order).
We consider the intersection interval $(a,b)\cap (\pi_1(d),\pi_1(c))=(p,q)$.

\begin{theorem}
Let $D$ be a finite area domain in $\mathbb{C}$ under the graph of an
arbitrary function $f:(A,B)\to[0,\infty]$ in the above sense. Then for any
quadruple $(a,b,c,d)$ of prime ends of $D$ in the given cyclic order, we have
$$
\lim_{\eps\to 0}\eps \cdot \m(\Gamma^{\eps}_{(a,b,c,d)}) =\int_p^q\frac{dx}{f(x)}=\m(\Gamma_{v})
$$
where $\Gamma_{v}$ is the family of vertical lines that connect the interval
$(a,b)$ to the interval $(c,d)$, and $[p,q]$ is the interval of the
$x$-coordinates of $\Gamma_{v}$.
\end{theorem}

\begin{proof}


Just like in the proof of Theorem \ref{thm:continuous} it is enough to show
that the inequality (\ref{continuous-is-nice}) holds even if $f$ is lower
semicontinuous.

We will use the well known fact that if $f:(A,B)\to[0,\infty]$ is a lower
semicontinuous function then there is a sequence of continuous functions
$f_n:(A,B)\to[0,\infty)$  such that $f_n(x)\leq f_{n+1}(x), n=1,2,\ldots$ and
$f_n(x)\to f(x)$ for every $x\in(A,B)$. Next, let
$$D_{n} = \{(x,y) :
x\in(A,B), y\in(0,f_n(x)) \}$$ and let $\G_{<\eta, n}$ be the collection of
curves $\g$ in $D_{n}$ such that $\g(0) \in (p,q)$, $\g(1)$ belongs to the graph of $f_n$ and
$|\pi_1(\g)|<\eta$. In other words, $\g\in\G_{<\eta, n}$ connects the
interval $(p,q)$ to the graph of $f_n$ and has ``horizontal variation"
$<\eta$.


Note that for every $\eta>0$ and every $n\in\mathbb{N}$ the family
$\G_{<\eta}$ overflows $\G_{<\eta,n}$. Indeed, since $f_n(x)<f(x)$ a curve
$\g\in\G_{<\eta}$ would have to ``hit" the graph of the continuous $f_n$
before ``reaching" the graph of $f$. Therefore $\m(\G_{<\eta}) \leq
\m(\G_{<\eta,n})$ and since $f_n$ is continuous Theorem \ref{thm:continuous}
yields
\begin{equation}
\limsup_{\eta\to0^{+}} \m(\G_{<\eta}) \leq \limsup_{\eta\to 0^{+}} \m(\G_{<\eta,n})
\leq \int_{p}^{q}\frac{dx}{f_n(x)}.
\end{equation}
Now, by our assumption $a,b,c$ and $d$ are different prime ends and therefore
$\m(\G_{(a,b,c,d)})<\infty$. Next we note that this implies that
$\min_{x\in[p,q]} f(x)>0$. Indeed, since $f$ is lower semicontinuous it
attain a minimum in $[p,q]$, we will denote this minimum by $m$. Now, since
$D$ is a domain we have that if $m=0$ then it is attained at one (or both) of
the endpoints of $[p,q]$. This would only be possible if (say) $a=d$, which
is a contradiction.

Now, since there is a constant $c>0$ such that $f(x)>c$ for every $x\in[p,q]$
then it is easy to see that we can assume that $f_n(x)\geq c>0$  for
$x\in(p,q)$ (just redefine the $f_n$ to be the maximum of $c$ and the old
$f_n$). Hence by the dominated convergence theorem we have $\int_p^q
(f_n(x))^{-1}dx\to \int_p^q (f_n(x))^{-1}dx$, which gives
(\ref{continuous-is-nice}) in the case of a semi-continuous $f$.
\end{proof}

The above theorem immediately gives Theorem \ref{thm:arbitrary} from
Introduction. Theorem \ref{thm:teichmuller} is a direct consequence of
Theorem 1 and Lemma \ref{lem:mod_liouville_measure}.

%

\end{document}